\documentclass[11pt,a4paper,oneside,reqno]{amsart}

\usepackage{latexsym}
\usepackage{amsmath}
\usepackage{amsthm}
\usepackage{amssymb}
\usepackage{amsfonts}
\usepackage{fancyhdr}
\usepackage{comment}

\usepackage{mathrsfs}                           

\newcommand{\R}{\mathbb{R}}                    
\newcommand{\abs}[1]{\lvert #1 \rvert}          
\newcommand{\norm}[1]{\lVert #1 \rVert}         

\newcommand{\id}{\mathrm{Id}}

\newcommand{\tS}{\mathcal{S}}
\newcommand{\ol}{\overline}
\newcommand{\tG}{\mathcal{G}}
\newcommand{\tC}{\mathcal{C}}

\theoremstyle{definition}
\newtheorem{thm}{Theorem}[section]
\newtheorem*{thmnonum}{Theorem}

\newtheorem{cor}[thm]{Corollary}

\newtheorem*{conjecture}{Conjecture}

\numberwithin{equation}{section}

\title[Nowhere conformally homogeneous manifolds]{Nowhere conformally homogeneous manifolds and limiting Carleman weights}

\author{Tony Liimatainen}
\address{Department of Mathematics and Systems Analysis, Aalto University}
\email{tony.liimatainen@tkk.fi}
\author{Mikko Salo}
\address{Department of Mathematics and Statistics, University of Helsinki}
\email{mikko.salo@helsinki.fi}

\date{November 8, 2010}

\begin{document}

\begin{abstract}
In this note we prove that a generic Riemannian manifold of dimension $\geq 3$ does not admit any nontrivial local conformal diffeomorphisms. This is a conformal analog of a result of Sunada concerning local isometries, and makes precise the principle that generic manifolds in high dimensions do not have conformal symmetries. Consequently, generic manifolds of dimension $\geq 3$ do not admit nontrivial conformal Killing vector fields near any point. As an application to the inverse problem of Calder\'on on manifolds, this implies that generic manifolds of dimension $\geq 3$ do not admit limiting Carleman weights near any point.
\end{abstract}

\maketitle


\section{Introduction} \label{section_introduction}

It is a well-known fact in geometry that generic Riemannian manifolds should not admit symmetries or other special structure. Various precise statements which formalize this principle may be found in the literature \cite{Ab, An, BaUr, Ber, KT, Uhl}. We mention here the result concerning nonexistence of local isometries proved by Sunada \cite{Su} in his work on isospectral manifolds.

\begin{thmnonum}
Let $M$ be a compact $C^{\infty}$ manifold without boundary. There is a residual set of Riemannian metrics on $M$ for which there are no isometries between any distinct open subsets of the manifold.
\end{thmnonum}

Recall that a subset of a topological space is called \emph{residual} if it contains a countable intersection of open dense sets. A property is called \emph{generic} if it holds for all elements in some residual set. We equip the set of all Riemannian metrics on $M$ with the $C^{\infty}$ topology as described in \cite{BaUr}, and this space then becomes a complete metric space. By Baire's theorem, any residual set in a complete metric space is dense. Thus, in particular, the set of metrics admitting no local isometries is dense in the set of all Riemannian metrics on $M$.

We are interested in a conformal version of Sunada's result. Here the two-dimensional case is special, since any Riemann surface is locally conformal to a subset of Euclidean space by the existence of isothermal coordinates and thus admits many nontrivial local conformal diffeomorphisms. It is an accepted fact in conformal geometry that generic Riemannian manifolds of dimension $n \geq 3$ should not admit conformal symmetries. However, we were not able to find a precise proof for such a statement in the literature. In this note we prove the following theorem which makes this statement rigorous.

\begin{thm} \label{thm_main}
Let $M$ be a compact boundaryless $C^{\infty}$ manifold having dimension $n \geq 3$. There is a residual set of Riemannian metrics on $M$ for which there are no conformal diffeomorphisms between any distinct open subsets of the manifold.
\end{thm}

Recall that a diffeomorphism $\phi$ from $(U,g)$ onto $(V,h)$ is called conformal if $\phi^* h = c g$ in $U$ for some smooth positive function $c$ on $U$. The function $c$ is called the conformal factor, and two manifolds are called conformal if there is a conformal diffeomorphism between them. The result states that for a generic Riemannian metric $g$ on $M$, there are no distinct open subsets $U$ and $V$ such that $(U,g)$ and $(V,g)$ would be conformal.

We now state some corollaries of the main theorem. A smooth vector field $X$ on a Riemannian manifold $(U,g)$ is called \emph{Killing} if the local flows generated by $X$ are isometries, and \emph{conformal Killing} if $X$ is Killing in $(U,cg)$ for some positive function $c$. Equivalently, $X$ is Killing (resp.~conformal Killing) if and only if $\mathcal{L}_X g = 0$ (resp.~$\mathcal{L}_X g = \lambda g$ for some function $\lambda$) where $\mathcal{L}_X g$ is the Lie derivative of $g$. If $g$ is given in local coordinates by the matrix $(g_{jk})$, then $\mathcal{L}_X g$ is the symmetric $2$-tensor with local coordinate expression  
$$
\mathcal{L}_X g = \sum_{i,j,k=1}^n (g_{jk} X^k_{;i} + g_{ik} X^k_{;j}) \,dx^i \otimes \,dx^j.
$$
Here $X^k_{;i} = \partial_{x_i} X^k + \Gamma_{il}^k X^l$ is the covariant derivative, where $\Gamma_{ab}^c$ are the Christoffel symbols for $g$.

It follows that local flows generated by conformal Killing vector fields are conformal diffeomorphisms, and the existence of a nontrivial conformal Killing field near a point implies the existence of a conformal diffeomorphism from some open set $(U,g)$ onto a disjoint open set $(V,g)$. Thus we have the following consequence of Theorem \ref{thm_main}.

\begin{cor} \label{cor_main2}
Let $M$ be a compact boundaryless $C^{\infty}$ manifold having dimension $n \geq 3$. There is a residual set of Riemannian metrics on $M$ which do not admit a nontrivial conformal Killing vector field near any point.
\end{cor}

Following \cite{NRS}, we say that a Riemannian manifold is \emph{locally conformally homogeneous} if for any point $p$ and for any tangent vector $v$ at $p$, there is a conformal Killing vector field $X$ in some neighborhood of $p$ such that $X(p) = v$. The last result implies that generic metrics in dimensions $n \geq 3$ are nowhere conformally homogeneous.

It is easy to convince oneself why Corollary \ref{cor_main2} could be true. If $\mathcal{L}_X g = \lambda g$, it follows by taking traces that $\lambda = \frac{2}{n} \text{div}(X)$ where $\text{div}(X) = \sum_{k=1}^n X^k_{;k}$ is the metric divergence. Thus conformal Killing vector fields are solutions of the equation $T(X) = 0$, where $T$ is the trace-free deformation tensor 
$$
T(X) = \mathcal{L}_X g - \frac{2}{n} \text{div}(X) g.
$$
Since $T(X)$ is trace-free and symmetric, in local coordinates $T(X) = 0$ is a system of $\frac{n(n+1)}{2}-1$ equations in $n$ unknowns. Therefore, if $n \geq 3$, it is not surprising that for generic metrics the only solution is $X = 0$.

A more sophisticated approach, described in \cite{BCEG} in a general setup, is to observe that the overdetermined system $T(X) = 0$ is of finite type. By the method of prolongation, one can construct a vector bundle $V$ together with a connection $\tilde{\nabla}$ such that solutions of $T(X) = 0$ correspond to parallel sections of $V$. However, the curvature of $\tilde{\nabla}$ imposes restrictions on the existence of parallel sections, and again for generic $g$ one expects to have no nontrivial solutions. Related results are given in \cite{Gov, GovS, Sem}.

The final corollary concerns recent results on Calder\'on's inverse problem on manifolds. This was the original motivation for the present study, and we explain the setup in some detail. Calder\'on's inverse problem was introduced in the Euclidean case by Calder\'on \cite{C} as the question of determining the conductivity function of a medium from voltage and current measurements made on its boundary, and it has been intensively studied (see the recent survey \cite{U_IP}). The anisotropic case corresponds to determining a conductivity matrix up to change of coordinates instead of a scalar function, and it was observed in \cite{LeU} that this is very closely related to a geometric problem which we proceed to describe.

Let $(M,g)$ be a compact oriented Riemannian manifold with smooth boundary and let $\Delta_g$ be the Laplace-Beltrami operator on $M$. 
The boundary data of harmonic functions on $(M,g)$ is given by the Cauchy data set 
$$
C_g = \{ (u|_{\partial M}, \partial_{\nu} u|_{\partial M}) \,;\, \Delta_g u = 0 \text{ in } M, u \in C^{\infty}(M) \}.
$$
Equivalently, the Cauchy data set is the graph of the Dirichlet-to-Neumann map 
$$
\Lambda_g: C^{\infty}(\partial M) \to C^{\infty}(\partial M), \ f \mapsto \partial_{\nu} u|_{\partial M} \text{ where } \Delta_g u = 0 \text{ in $M$, $u|_{\partial M} = f$}.
$$
Here $\partial_{\nu} u|_{\partial M} = g(du,\nu)|_{\partial M}$ where $\nu$ is the $1$-form corresponding to the outer unit normal vector on $\partial M$.

It follows immediately that $C_{\psi^* g} = C_g$ whenever $\psi: M \to M$ is a diffeomorphism satisfying $\psi|_{\partial M} = \id$. The Calder\'on problem on manifolds asks to prove that the Riemannian metric on a manifold is determined up to diffeomorphism by the Cauchy data of harmonic functions (or, equivalently, by the Dirichlet-to-Neumann map).

\begin{conjecture}
Let $(M,g_1)$ and $(M,g_2)$ be two compact oriented Riemannian manifolds with smooth boundary having dimension $n \geq 3$. If $C_{g_1} = C_{g_2}$, then $g_2 = \psi^* g_1$ for some diffeomorphism $\psi: M \to M$ satisfying $\psi|_{\partial M} = \id$.
\end{conjecture}

This has been proved for real-analytic manifolds \cite{LTU, LaU, LeU} and also for Einstein manifolds \cite{GS}, and the corresponding two-dimensional result is known for any smooth Riemann surface \cite{LaU}. The general case in dimensions $n \geq 3$ is open. Recently in \cite{DKSaU, KSaU} progress was made on this problem for a certain class of conformal smooth manifolds. The results were based on a construction of special harmonic functions involving \emph{limiting Carleman weights}, introduced in \cite{KSU} in the Euclidean case. We refer to \cite{DKSaU} for the precise definition. For the purposes of this paper, we only mention that a limiting Carleman weight is essentially a smooth real-valued function $\varphi$ with nonvanishing gradient for which one expects to have weighted estimates 
$$
\norm{u}_{L^2} \leq \frac{C}{\abs{\tau}} \norm{e^{\tau \varphi} \Delta_g e^{-\tau \varphi} u}_{L^2}, \quad u \in C^{\infty}_c,
$$
which hold for all real $\tau$ with $\abs{\tau}$ sufficiently large and where $C$ is independent of $\tau$ and $u$.

Let $(U,g)$ be an open Riemannian manifold (that is, $U$ has no boundary and no component is compact). It was proved in \cite[Theorem 1.2]{DKSaU} that a simply connected open manifold $(U,g)$ admits a limiting Carleman weight if and only if $(U,cg)$ admits a parallel unit vector field for some positive function $c$. Since a parallel vector field is Killing, the existence of a limiting Carleman weight implies the existence of a nontrivial conformal Killing vector field. Corollary \ref{cor_main2} then implies the following result, which was stated informally in \cite{DKSaU} without proof.

\begin{cor}
Let $(U,g)$ be an open submanifold of some compact manifold $(M,g)$ without boundary, having dimension $n \geq 3$. There is a residual set of Riemannian metrics on $M$ which do not admit limiting Carleman weights near any point of $U$.
\end{cor}

\section{Proof of Theorem~\ref{thm_main}}
We now prove Theorem~\ref{thm_main}. Denote by $\tG$ the set of all Riemannian metrics on $M$ for which there are no conformal diffeomorphisms between any distinct open subsets of $M$. The statement of Theorem~\ref{thm_main} is that $\tG$ is a countable intersection of sets which are open and dense in the $C^\infty$ topology of Riemannian metrics on $M$.

We divide the proof in parts. First we show that, by choosing a suitable basis of open sets in $M$, the problem can be localized and reduces to finding suitable metrics for open balls in $\R^n$. Subsequently, we argue that the open balls admit a residual set of metrics for which there are no conformal diffeomorphisms between the balls. This involves using compactness and regularity results for conformal diffeomorphisms, and a dimension count argument based on jet spaces. The proof is a conformal analogue of the proof of Sunada's result \cite[Proposition 1]{Su} mentioned in the introduction.

\begin{proof}[Proof of Theorem~\ref{thm_main}]

I. \emph{Basis}: We choose a basis $\{U_i\}_{i=1}^{\infty}$ for the topology of $M$ such that the closure $\overline{U}_i$ of each $U_i$ is diffeomorphic to a closed ball in $\R^n$. Let us denote by $\tS_{ik}$ the set of metrics $g$ on $M$ for which there exists a conformal diffeomorphism $\phi:(\ol{U}_i,g)\rightarrow (\ol{U}_k,g)$. We will show that 
\begin{equation*}
 \tG = \bigcap_{i,k \,;\, \ol{U}_i \cap \ol{U}_k = \emptyset}\tC\mathcal{S}_{ik}.
\end{equation*}

It is clear that the set $\tG$ is contained in the complement of each $\tS_{ik}$, $i\neq k$. For the other direction, let a metric $g$ belong to the complement of $\tG$. Then there exists a conformal diffeomorphism $\phi\neq \mbox{Id}$ between some distinct open sets $U$ and $V$ in $M$. By changing $\phi$ to its (conformal) inverse if necessary, we can find a point $x\in U \setminus V$ with $\phi(x)\neq x$. Restricting $\phi$ to sufficiently small neighbourhoods of $x$ and $\phi(x)$ we have $g\in \tS_{ik}$ where $\ol{U}_i \cap \ol{U}_k = \emptyset$. Thus $\tC\mathcal{G}\subset \bigcup_{i,k \,;\, \ol{U}_i \cap \ol{U}_k = \emptyset}\mathcal{S}_{ik}$.


II. \emph{$\tS_{ik}$ is a countable union of closed sets}:
We fix open sets $U_i, U_k$ with $\ol{U}_i \cap \ol{U}_k = \emptyset$, and write $U=U_i$, $V=U_k$ and $\tS = \tS_{ik}$. If a metric belongs to $\tS$, there is at least one conformal mapping between $U$ and $V$ and a corresponding conformal factor. We represent the set $\tS$ as a countable union 
\begin{equation*}
 \tS=\bigcup_{N=1}^\infty\tS^N,
\end{equation*}
where, for some fixed integer $s > \frac{n}{2} + 2$, 
\begin{equation*}
 \tS^N:=\{g \in \tS: \exists \ \phi:\ol{U}\rightarrow \ol{V}, \ cg=\phi^*g, \ \norm{c}_{W^{s,2}(U)} + \norm{1/c}_{W^{s,2}(U)} \leq N \}.
\end{equation*}
Here, since $\ol{U}$ is diffeomorphic to the closed unit ball in $\R^n$, we can use the standard Sobolev norms in Euclidean space, 
$$
\norm{f}_{W^{s,2}(U)} = \sum_{\abs{\alpha} \leq s} \norm{\partial^{\alpha} f}_{L^2(U)}.
$$

Suppose $(g_n)$ is a sequence of metrics in $\tS^N$ converging in the $C^\infty$ topology to some metric $g_0$. We will prove that the limit metric $g_0$ belongs to $\tS^N$, showing that $\tS^N$ is closed as required.

Let $\phi_n$ be a conformal diffeomorphism from $\ol{U}$ onto $\ol{V}$ as in the definition of $\tS^N$, with a conformal factor $c_n$ so that $\phi_n^* g_n = c_n g_n$, $\norm{c_n}_{W^{s,2}(U)} \leq N$. The bounded sequence $(c_n)$ has a subsequence (also denoted by $(c_n)$) converging weakly to some $c_0$ in $W^{s,2}(U)$. By compact Sobolev embedding $c_n \to c_0$ strongly in $C^2(\ol{U})$. We define new metrics $h_n$ in $\ol{U}\cup\ol{V}$ by
\begin{equation*}
h_n:=\left\{
\begin{array}{rl}
c_n g_n, & x\in \ol{U}\\
g_n, & x\in \ol{V} \\
\end{array}\right.
\end{equation*}
and use a similar definition for $h_0$. The metrics $h_n$ and $h_0$ are $C^\infty$ and $C^2$ smooth, respectively. Moreover, for each $n\in\mathbb{N}$, $\phi_n$ is an isometry from $(\ol{U},h_n)$ onto $(\ol{V},h_n)$.

Consider the distance functions $d_n = d_{h_n}$ and $d_0 = d_{h_0}$. The definition of the distance functions and a computation in local coordinates show that for $x,y \in \ol{U}$ or for $x,y \in \ol{V}$, 
\begin{equation*}
 C_n^{-1} d_n(x,y)\leq d_0(x,y)\leq C_n d_n(x,y)
\end{equation*}
where $C_n \geq 1$, and the convergence $h_n\rightarrow h_0$ in $C^2$ implies that $C_n \to 1$. Moreover, the isometry $\phi_n:(\ol{U},h_n) \rightarrow (\ol{V},h_n)$ is distance preserving, and we obtain 
\begin{equation*}
d_0(\phi_n(x),\phi_n(y))\leq C_n d_n(\phi_n(x),\phi_n(y))=C_n d_n(x,y)\leq C_n^2 d_0(x,y).
\end{equation*}
That is, the sequence $(\phi_n)$ is equicontinuous with respect to $d_0$. The Arzela-Ascoli theorem yields a uniformly converging subsequence, $\phi_n\rightarrow \phi_0$.

Similar considerations for $\phi_n^{-1}$, using now the bound $\norm{1/c_n}_{W^{s,2}(U)} \leq N$,  yield that $\phi_0$ is a homeomorphism. It is also distance preserving, since 
\begin{equation*}
\begin{split}
 \frac{1}{C_n^2}d_0(x,y)&=\frac{1}{C_n^2}d_0(\phi_n^{-1}\phi_n(x),\phi_n^{-1}\phi_n(y)) \\
 &\leq d_0(\phi_n(x),\phi_n(y)) \leq C_n^2 d_0(x,y)
\end{split}
\end{equation*}
where $C_n\rightarrow 1$.  A distance preserving homeomorphism $\phi_0$ from $U$ onto $V$, which have $C^2$ smooth metric, is itself $C^3$ smooth by Theorem 2.1 of~\cite{T_isometries}. By the same theorem, $\phi_0$ is in fact an isometry in the sense that
\begin{equation*}
 h_0=\phi_0^*h_0.
\end{equation*}
That is,
\begin{equation*}
 c_0g_0=\phi_0^*g_0.
\end{equation*}
This implies that $\phi_0$ is a $C^3$ conformal diffeomorphism between two $C^{\infty}$ manifolds. Now, a theorem of Ferrand~\cite{Ferrand} shows that $\phi_0$ and also $c_0$ are actually infinitely smooth. Thus the metric $g_0$ belongs to $\tS^N$, so $\tS^N=\tS^N_{ik}$ is closed for all $i, k$ with $\ol{U}_i \cap \ol{U}_k = \emptyset$ and for all $N$.

III. \emph{$\tC \tS_{ik}^N$ is dense}: 
So far we have proved that 
$$
 \tG = \bigcap_{i,k,N \,;\, \ol{U}_i \cap \ol{U}_k = \emptyset}\tC\mathcal{S}_{ik}^N
$$
where each $\tC\mathcal{S}_{ik}^N$ is open. In the final step of the proof we show that $\tC\mathcal{S}_{ik}^N$ is dense for each $i$, $k$, $N$. In fact, it is sufficient to prove that any metric in $\mathcal{S}_{ik}$ can be approximated arbitrarily well by metrics which are not in $\mathcal{S}_{ik}$. This step concludes the proof.

We write $\mathcal{S}=\mathcal{S}_{ik}$, $U = U_i$ and $V = U_k$. We may assume $\ol{U}$ and $\ol{V}$ are disjoint closed balls in $\R^n$, with $0\in U$. Let $g\in \mathcal{S}$ and let $\phi$ be a corresponding conformal diffeomorphism $(\ol{U},g_1)\rightarrow (\ol{V},g)$, where the restrictions of the metric $g$ to $\ol{U}$ and $\ol{V}$ are denoted by $g_1$ and $g$, respectively. Let us also write $x=\phi(0)\in V$.

If $\tau_x$ is the translation $y\mapsto y+x$, then the mapping $\psi:=\tau_x^{-1}\circ \phi$ is a diffeomorphism defined in the neighborhood $\ol{U}$ of the origin and it preserves the origin, $\psi(0)=0$. By the conformality of $\phi$, the metric in $\ol{U}$ is of the form
\begin{equation}\label{metricform}
g_1=c\phi^*(g)=c\psi^*(\tau_x^*(g)),
\end{equation}
for some $c\in C^\infty(\ol{U})$.

The previous equation shows that the form of a metric in $\mathcal{S}$ is restricted near the origin. To study the nature of the restriction~\eqref{metricform}, we calculate upper bounds for the dimensions of the jet spaces \cite{Saunders} of metrics of the form~\eqref{metricform}. We will see that, for large enough $k$, the dimension of the $k^{th}$ jet space of all local metrics is larger than the dimension of the $k^{th}$ jet space of metrics of the form~\eqref{metricform}. 


Metrics on a neighbourhood of the origin are mappings from $U$ to the space of symmetric positive definite bilinear forms on $\R^n$ , which has dimension $n(n+1)/2$. We denote by $\mathcal{M}^k$ the $k^{th}$ jet space of such mappings at the origin. Accordingly, the dimension of $\mathcal{M}^k$ is
\begin{equation*}
\frac{n(n+1)}{2}\binom{n+k}{k}.
\end{equation*}

Let us then consider the jet spaces of metrics of the form~\eqref{metricform}. Let $\mathcal{D}^k$ be the $k^{th}$ jet space of local diffeomorphisms $\psi$ at the origin with $\psi(0) = 0$. Its dimension is
\begin{equation*}
 n\binom{n+k}{k}-n.
\end{equation*}

By the chain rule, we see that the $k$-jet of $\psi^*g$ at the origin is determined by the $(k+1)$-jet of $\psi$ and the $k$-jet of $g$ at the origin. In this way, $\mathcal{D}^{k+1}$ defines a smooth action
\begin{equation*}
G_1: \mathcal{D}^{k+1}\times \mathcal{M}^k\rightarrow \mathcal{M}^k.
\end{equation*}
Similarly $\mathcal{C}^k$, the $k^{th}$ jet space of (positive) functions at the origin, defines a smooth action in $\mathcal{M}^k$,
\begin{equation*}
G_2: \mathcal{C}^k\times\mathcal{M}^k\rightarrow \mathcal{M}^k. 
\end{equation*}
Moreover, we define a mapping $\gamma: V \rightarrow \mathcal{M}^k$ as the $k$-jet at the origin of the pullback of $g$ by $\tau_y$:
\begin{equation}
y\mapsto (\tau_y^*g)^{(k)}.
\end{equation}
Combining the above defined mappings naturally, we get a smooth mapping $G:\mathcal{C}^k\times\mathcal{D}^{k+1}\times V\rightarrow \mathcal{M}^k$,
\begin{equation*}
(c^{(k)},\psi^{(k+1)},y)\mapsto G_2\left(c^{(k)},G_1(\psi^{(k+1)},\gamma(y))\right).
\end{equation*}
By \eqref{metricform}, the $k^{th}$ jet of $g_1$ at the origin is given by $G(c^{(k)},\psi^{(k+1)},x)$ where $c^{(k)}$ and $\psi^{(k+1)}$ are the corresponding jets of $c$ and $\psi$Ê at the origin.

By Sard's theorem~\cite{Sard}, the image of the mapping $G$ has Lebesgue measure zero if the rank of $G$ is smaller than the dimension of the target space. The rank of the mapping $G$ is at most the dimension of its domain, which is
\begin{equation*}
\binom{n+k}{k}+n\binom{n+k+1}{k+1}-n+n=\binom{n+k}{k}+n\binom{n+k+1}{k+1}.
\end{equation*}
This is smaller than the dimension of the target space $\mathcal{M}^k$, when $k\geq 4$ and $n\geq 3$. Since we assume $n\geq 3$, the image of $G$ is of measure zero in $\mathcal{M}^k$, and thus its complement is dense, when $k\geq 4$. Thus, using a $C^\infty$ bump function we can deform the original metric near the origin such that the modified metric satisfies the following conditions: 
\begin{itemize}
\item it is arbitrarily close to the original metric in the $C^\infty$ topology
\item it is unaltered in $V$
\item the $k$-jets of the modified metric at the origin are not in image of $G$, for some $k\geq 4$. 
\end{itemize}
We have shown that the complement of $\mathcal{S}$ is dense in the $C^\infty$ topology of Riemannian metrics of $M$. This concludes the proof.
\end{proof}
\subsection*{Acknowledgements}

T.L.~is partly supported by Finnish National Graduate School in Mathematics and its Applications, and M.S.~is supported in part by the Academy of Finland. The authors would like to thank Rod Gover and Robin Graham for helpful suggestions.


\providecommand{\bysame}{\leavevmode\hbox to3em{\hrulefill}\thinspace}
\providecommand{\href}[2]{#2}

\end{document}